\documentclass[11pt,a4paper, reqno]{amsart}

\usepackage{amssymb, amsthm, amsmath, dsfont, lmodern, mathrsfs, graphicx, xspace}

\usepackage[sorting=nyt,backend=biber,style=numeric,isbn=false,alldates=year,url=false,doi=false,eprint=false,maxalphanames=4,minalphanames=3,maxbibnames=99]{biblatex}
\renewbibmacro*{note+pages}{\printfield{note}}
\bibliography{ref}

\newtheorem{theorem}{Theorem}
\newtheorem{lemma}{Lemma}
\newtheorem{remark}{Remark}

\newcommand\R{{\ensuremath {\mathbf R} }}
\newcommand{\supp}{{\rm supp}}
\newcommand\Pro{{\ensuremath {\mathbb P} }}
\newcommand\N{{\ensuremath {\mathbf N} }}
\newcommand\Q{{\ensuremath {\mathbb Q} }}

\newcommand{\cP}{\mathcal{P}}
\newcommand{\eps}{\varepsilon}
\renewcommand{\d}{{\rm d}}
\renewcommand{\tilde}{\widetilde}
\DeclareMathOperator*{\essinf}{ess\,inf}
\newcommand{\clo}{c_{\rm LO}}
\newcommand{\dd}{\text{d}}
\newcommand{\cs}{1 + \frac sd}

\title[Absence of non-compactly supported minimisers for Lieb--Oxford]{Absence of non-compactly supported minimisers for the Lieb--Oxford bound}
\author{Simone Di Marino}
\address{DIMA \& MaLGa, Università di Genova, Genova, Italy}
\email{simone.dimarino@unige.it}
\author{Rodrigue Lelotte}
\address{SAMM, Universit\'e Paris 1 Panth\'eon-Sorbonne, Paris, France}
\email{rodrigue.lelotte@univ-paris1.fr}
\date{\today}

\begin{document}
	\begin{abstract}
		We prove the minimisers of the Lieb-Oxford bound (if any) for a fixed (and finite) number of particles $N \geq 1$ are necessarily compactly supported, extending a result proved by E. H. Lieb and S. Oxford in the one-particle case.  
	\end{abstract}
    \maketitle
    \tableofcontents
	
	\section{Introduction}
	The \emph{Lieb-Oxford inequality}  is a fundamental lower bound for the \textit{indirect part} of the Coulomb or more generally Riesz energy of a system of charged particles \cite{seidl_lieb-oxford_2022, lieb_stability_2009}. It asserts the existence of a constant $\clo(s, d) \geq 0$ depending on the dimension $d$ and a homogeneity parameter $0 < s < d$ such that for all probability measures $\Pro$ on $(\R^d)^N$
	\begin{multline}\label{eq:LO}\tag{LO}
		\int_{\R^{dN}} \sum_{1 \leq i < j \leq N} \frac{1}{|x_i - x_j|^{s}} \d \Pro(x_1, \dots, x_N) - \frac{1}{2}\iint_{\R^d \times \R^d} \frac{\rho_\Pro(x)\rho_\Pro(y)}{|x - y|^s} \dd x \dd y\\
		\geq -\clo(s,d) \int_{\R^d} \rho_\Pro^{1 + \frac{s}{d}}
	\end{multline}
	where $\rho_\Pro$ is the sum of the marginals of $\Pro$ with respect to each copy of $\R^d$ in the Cartesian product $(\R^d)^N$ --- we refer to \eqref{eq:one-particle-density} below for a precise definition of $\rho_\Pro$. Here, it is assumed that $\rho_\Pro$ is absolutely continuous with respect to the Lebesgue measure and that $\rho_\Pro \in L^1 \cap L^{\cs}(\R^d)$. In particular, $\rho_\Pro \geq 0$ and $\int_{\R^d} \rho_\Pro = N$. We emphasise that the constant $\clo(s, d)$ does \textit{not} depend on $N$, which here represents the number of particles. It is therefore said to be \textit{universal} in this sense. The Lieb-Oxford inequality \eqref{eq:LO} can be restated as
	\begin{equation}\label{eq:LO-rho}\tag{LO'}
		\boxed{C(\rho) - \tfrac12 D(\rho) \geq - \clo(s, d) \int_{\R^d} \rho^{\cs}}
	\end{equation}
	where $C(\rho)$ is obtained from the first term on the left-hand side in \eqref{eq:LO} by minimising over all $\Pro$'s such that $\rho_\Pro = \rho$ for a fixed $\rho \geq 0$ with $\int_{\R^d} \rho \in \N^*$— see below in \eqref{eq:MOT} for an exact definition of $C(\rho)$. This term corresponds mathematically to a \textit{multimarginal optimal transport problem}. It physically corresponds to the lowest \textit{Riesz energy} of a system of $N$ particles $x_1, \dots, x_N \in \R^d$ distributed according to $\Pro$ under the constraint that the one-body density $\rho_{\Pro}$ is given by $\rho$. In the quantum setting, $\Pro = |\Psi|^2$ where $\Psi \in L^2(\R^{dN})$ is the underlying wavefunction of system of $N$ particles. The quantity $$D(\rho) :=\iint_{\R^d \times \R^d} \frac{\rho(x)\rho(y)}{|x - y|^s} \dd x \dd y$$ corresponds to the \textit{Hartree energy} of $\rho$, that is the total classical energy associated with a charge distributed in space according to $\rho$.

  	This bound was originally proved by Lieb \cite{lieb_lower_1979} and Lieb and Oxford \cite{lieb_improved_1981} in the case of the Coulomb interaction in dimension $d=3$, that is for the choice $s = 1$. It was later extended to any dimension $d\geq1$ and for all $0 < s < d$. We refer the reader to \cite[Sec. 5.3]{lewinUniversalFunctionalsDensity2023} and references therein. In \eqref{eq:LO}, we stress that we consider $\clo(s, d)$ to be the \underline{smallest possible} constant such that the inequality holds. This inequality is both of theoretical and practical importance, notably in the context of \emph{Density Functional Theory} \cite{engel_density_2011, burke_dft_2013} where it is used to calibrate \textit{exchange-correlation functionals} \cite{levy_tight_1993, perdew_generalized_1996, perdew_lieb-oxford_2022, sun_accurate_2016}. As such, the search for the optimal value of $\clo(s, d)$ has retained a lot of attention since the original work of E.H. Lieb and S. Oxford, see \textit{e.g.} \cite{seidl_challenging_2016, kin-lic_chan_optimized_1999, seidl_lieb-oxford_2022, lewin_improved_2015, lewin_improved_2022}.

	As already stressed above, in the inequality \eqref{eq:LO} the constant $\clo(s, d)$ is universal in the number of particles in the system. One can nevertheless consider the Lieb--Oxford bound for a \textit{fixed} (and \textit{finite}) number of particles $N \in \N^*$ and look for the smallest constant such that \eqref{eq:LO} is verified for all $N$-particle distributions $\Pro \in \cP((\R^d)^N)$. We denote by $\clo(s, d, N)$ this constant -- and later on simply by $\clo$ for shortness if no ambiguity arises. 
	
	In this paper, \textit{we prove that any minimiser (if any) at fixed $N \geq 1$ of the \eqref{eq:LO} bound, that is a density $\rho \in L^1 \cap L^{\cs}(\R^d, \R^+)$ with $\int_{\R^d} \rho = N$ which attains the smallest possible constant $\clo(s, d, N)$ in \eqref{eq:LO-rho} among all such densities, is necessarily compactly supported}. This was already known in the very special case $N = 1$, as proved by Lieb and Oxford \cite{lieb_improved_1981}. A heuristic and mathematically not sound argument showing that this should hold for any number of particles $N\in \N^*$ was given in the chemistry literature, see \cite{seidl_challenging_2016}. Altogether, our main result reads as follows.
	
	\begin{theorem}\label{thm:main-thm}
		Let $0 < s < d$ and let $\rho \in L^1 \cap L^{\cs}(\R^d, \R_+)$ with $\int_{\R^d}\rho = N \geq 1$ be a minimiser (if any) for the Lieb-Oxford bound \eqref{eq:LO-rho} for a fixed number of particles $N \geq 1$. Then, $\rho$ must be compactly supported. 
	\end{theorem}
	A crucial element in the proof is the asymptotic behaviour of the Kantorovich potential associated to the multimarginal optimal transport problem, which was conjectured in \cite{seidl_strictly_2007, frieseckeStrongInteractionLimitDensity2023} and proved recently in \cite{lelotteAsymptoticsKantorovichPotential2025}. This is recalled below in Lemma~\ref{lem:up-bnd-asymp}. 
	\textbf{Beware that we do not state the existence of a minimiser, which is an open question to the best of our knowledge}. The proof of Theorem~\ref{thm:main-thm} is given in Section~\ref{sec:proof-main-thm}. In Section~\ref{sec:duality-theory}, we recall results on the duality theory for the optimal transport problem with Riesz costs which will be crucial for the proof of our main result. 
\section{Properties of Kantorovich potentials}\label{sec:duality-theory}
The first term in the Lieb-Oxford bound \eqref{eq:LO-rho} corresponds to a multimarginal optimal transport problem. Indeed, we have
\begin{equation}\label{eq:MOT}\tag{OT}
	C(\rho) := \min_{\rho_\Pro = \rho} \left\{\int_{\R^{dN}} c(x_1, \dots, x_N) \d \Pro(x_1, \dots, x_N)\right\}
\end{equation}
where we let 
\begin{equation*}
	c(x_1, \dots, x_N) = \sum_{1 \leq i < j \leq N} \frac{1}{|x_i - x_j|^s}.
\end{equation*}
Here $\rho$ is a density, \textit{i.e.} $\rho \geq 0$ with $\int_{\R^d}\rho = N \in \N^*$ and the minimum above runs over all $N$-particle distributions $\Pro \in \cP((\R^d)^N)$ such that $\rho_\Pro$ is equal to $\rho$, where $\rho_\Pro$ is the sum of the marginals of $\Pro$ with respect to each copy of $\R^d$ in the Cartesian product $(\R^d)^N$:
\begin{equation}\label{eq:one-particle-density}
	\rho_\Pro(x) = \sum_{i = 1}^N \int_{(\R^d)^{N-1}} \Pro(\d x_1, \dots, \d x_{i-1}, x, \d x_{i+1}, \dots, \d x_N).
\end{equation}
In fact, because $c$ is symmetric with respect to the particles, we may without loss of generality restrict our attention to symmetric $\Pro$'s in the problem \eqref{eq:MOT} in which case all these marginals are equal and the constraint simply reads that the first (or in fact, any) marginal of $\Pro$ be given by $\rho/N$. Therefore \eqref{eq:MOT} is a multimarginal optimal transport problem where all the marginals are given by $\rho/N$ and the cost of transportation $c$ is the so-called \textit{Riesz cost}. We refer the reader to \cite{frieseckeStrongInteractionLimitDensity2023} for an extensive survey of this optimal transport problem \eqref{eq:MOT}, or yet \cite{marino_9_2017}. 

Let us now recall useful results regarding the duality theory for the optimal transport problem \eqref{eq:MOT}. It is well-known, see \cite{buttazzo_continuity_2018} and also \cite{colombo_continuity_2019}, that the optimal transport problem admits a dual formulation, called the \textit{Kantorovich dual}, and which reads
\begin{equation}\label{kd}\tag{KD}
	C(\rho) = \sup_{\phi} \, \left\{ \int_{\R^d}\phi \rho  \right\}
\end{equation} 
where the \textit{supremum} runs over all $\rho$-integrable functions $\phi : \R^d \to \R$  such that the inequality
\begin{equation}\label{eq:constraint}
	\sum_{i = 1}^N \phi(x_i) \leq c(x_1, \dots, x_N)
\end{equation}
holds everywhere.  It is proved in \cite{buttazzo_continuity_2018, colombo_continuity_2019} that, under a \textit{small concentration} assumption on the target marginal $\rho$ -- see more precisely \cite[Assumption (A)]{buttazzo_continuity_2018} -- there exists a maximiser for the formulation \eqref{kd}, so-called \textit{Kantorovich potential} (KP). In the case where $\rho$ is absolutely continuous, this assumption is immediately verified. Moreover, it is shown in \cite{buttazzo_continuity_2018} that a KP for \eqref{kd} can be chosen to verify the equation
\begin{equation}\label{eq:c-conj}
	\phi(x) = \inf_{x_2, \dots, x_N \in \R^d} \left\{c(x, x_2,\dots, x_N) - \sum_{i = 2}^N \phi(x_i)\right\}.
\end{equation}
Such a potential is said to be \textit{$c$-conjugated}. Furthermore, it is proved in \cite{lelotteExternalDualCharge2024} that KPs for \eqref{kd} are in fact unique up to additive constant on the connected components of the target density $\rho$ and necessarily Lipschitz-continuous on the support of $\rho$.

In this paper, \textbf{we will work with a slightly different formulation of the Kantorovich duality \eqref{kd}.} First, we relax the inequality \eqref{eq:constraint} and require that it be verified only \textit{almost} everywhere with respect to the Lebesgue measure. In this direction, let us define
\begin{equation}\label{eq:E}
	E(\phi) := \essinf_{x_1, \dots, x_N \in \R^d} \left\{ c(x_1, \dots, x_N) - \sum_{i = 1}^N \phi(x_i) \right\}
\end{equation}
where the essential infimum is understood with respect to the Lebesgue measure. Physically, $E(\phi)$ essentially corresponds to the minimal energy that can be reached by a system of $N$ classical particles which interact through the Riesz interaction potential and that are subjected to an external potential given by \textit{minus} $\phi$.  Then, remark that for \textit{any} function $\phi$ (even defined only almost everywhere) then $\phi + \frac{E(\phi)}{N}$ verifies the constraint \eqref{eq:constraint} almost everywhere. We then claim that the Kantorovich duality \eqref{kd} can be reformulated as
\begin{equation}\label{kd-2}\tag{KD'}
	C(\rho) = \sup_{\phi \in \mathcal I_\rho}\left\{\int_{\R^d} \phi \rho + E(\phi)\right\}
\end{equation}
where we let $\mathcal{I}_\rho := L^1_{loc}(\R^d) \, \cap \, L^1(\rho)$. \textbf{We extend the notion of \textit{Kantorovich potential} to mean any $\phi \in \mathcal{I}_\rho$ which maximises the above problem \eqref{kd-2}.} The fact that \eqref{kd} and \eqref{kd-2} are equivalent is rather trivial if one restricts the supremum in \eqref{kd-2} to run as in \eqref{kd} over the set of $\rho$-integrable functions $\phi : \R^d \to \R$ and if we substitute an infimum to the essential infimum in \eqref{eq:E}. Nevertheless, we emphasise that in \eqref{kd-2} we work with KPs which are elements of a Lebesgue space and thus only defined \textit{a priori} almost everywhere, whereas in the standard formulation of the Kantorovich duality \eqref{kd} KPs are assumed implicitly from the start to be defined everywhere. The reason to distinguish at first these two settings in that in our proof of Theorem~\ref{thm:main-thm}, our candidate potential (denoted there $\phi_\rho$) will arise from a variational condition which places it naturally in a Lebesgue space. The fact that the non-standard strong Kantorovich duality \eqref{kd-2} holds is discussed and proved in Appendix~\ref{app:extension-buttazo}. \textit{A posteriori}, the distinction between these two settings may seem of no use since we have the following result:
\begin{theorem}\label{thm:summary}
	Let $\rho \in L^1(\R^d, \R_+)$ with $\int_{\R^d} \rho = N$. Then:
	\begin{enumerate}
		\item\label{item:existence} There exists at least one maximiser $\phi \in \mathcal{I}_\rho$ to the Kantorovich dual \eqref{kd-2} -- a so-called Kantorovich potential. Moreover, $\phi$ can be chosen (representable as a function which is) $c-$conjugated, Lipschitz-continuous and also uniformly bounded on $\R^d$.
        \item If $\psi \in \mathcal{I}_\rho$ is any Kantorovich potential for \eqref{kd-2}, then there exists another Kantorovich potential $\phi$  (representable as a function which is) $c-$conjugated, Lipschitz continuous and uniformly bounded so that 
        $$
        \psi + \frac{E(\psi)}{N} \leq \phi
        $$
        where the inequality is understood almost everywhere. Furthermore, by optimality it holds that $\psi + \frac{E(\psi)}{N} = \phi$ almost everywhere on the support of $\rho$. 
	\end{enumerate}
\end{theorem}
The proof of Theorem~\ref{thm:summary} is given in Appendix~\ref{app:extension-buttazo}. We now show in the next lemma that the Kantorovich potentials for the dual problem \eqref{kd-2} are characterised in a variational way as follows.
\begin{lemma}\label{lem:subgradient}
Let $\phi \in \mathcal{I}_\rho$ where we recall that $\mathcal{I}_\rho = L^1_{loc}(\R^d) \, \cap \, L^1(\rho)$. If for all compactly supported $\eta \in L^\infty(\R^d)$ it holds that
\begin{equation}\label{eq:subdifferential}
	C(\eta) \geq C(\rho) + \int_{\R^d} \phi \, \d (\eta - \rho),
\end{equation}
then $\phi$ is a Kantorovich potential for \eqref{kd-2}.
\end{lemma}
\begin{proof}[Proof of Lemma~\ref{lem:subgradient}]	
	Let $\phi \in \mathcal{I}_\rho$ verify the inequality \eqref{eq:subdifferential}. Let $x_1, \dots, x_N \in \R^d$, and assume that $x_i \neq x_j$ for all $i \neq j$. Let then $0 < \eps < \min_{i \neq j} |x_i - x_j|$, and let us define 
	\begin{equation}
		\eta_\eps := |B_\eps|^{-1}\sum_{i = 1}^N \mathds{1}_{B_\eps(x_i)},
	\end{equation}
	where $B_\eps(x_i)$ is the ball of $\R^d$ centred at $x_i$ and of radius $\eps$ and $|B_\eps|$ is the volume of such a ball. Now, let us define the trial state
	\begin{equation}
		\Q_\eps := \frac{1}{N!} \sum_{\sigma \in \mathfrak{S}_N} \frac{\mathds{1}_{B_\eps(x_{\sigma(1)})}}{|B_\eps|} \otimes \cdots \otimes \frac{\mathds{1}_{B_\eps(x_{\sigma(N)})}}{|B_\eps|}.
	\end{equation}
	It is seen that the one-body density of the $N$-particle distribution $\Q_\eps$ verifies $\rho_{\Q_\eps} = \eta_\eps$. By definition of the optimal transport problem, it holds that
	\begin{equation}
		C(\eta_\eps) \leq \int_{\R^{dN}} c(x_1, \dots, x_N) \d \Q_\eps
	\end{equation}
    By plugging $\eta_\eps$ into \eqref{eq:subdifferential}, we obtain that 
	\begin{equation}\label{eq:ineq-eta-eps}
		C(\rho) \leq C(\eta_\eps) - \int_{\R^d} \phi \eta_\eps + \int_{\R^d} \phi \rho.
	\end{equation}
	By continuity of the Riesz cost $c$ away from the diagonals, we have
	\begin{equation}
		\limsup_{\eps \to 0} C(\eta_\eps) \leq c(x_1, \dots, x_N).
	\end{equation}
	 Furthermore, by local integrability of $\phi$, the standard Lebesgue differentiation theorem yields
	\begin{equation}
		\int_{\R^d} \phi \eta_\eps \xrightarrow[\eps \to 0]{} \sum_{i = 1}^N \phi(x_i)
	\end{equation}
	for almost all $x_1, \dots, x_N \in \R^d$. Therefore, we obtain letting $\eps \to 0$ in the inequality \eqref{eq:ineq-eta-eps} that
	\begin{equation}
		C(\rho) \leq c(x_1, \dots, x_N) - \sum_{i = 1}^N \phi(x_i) + \int_{\R^d} \phi \rho
	\end{equation}
	for almost all $x_1, \dots, x_N \in \R^d$ (distinct from one another). Therefore, taking the essential infimum with respect to $x_1, \dots, x_N \in \R^d$, we obtain
	\begin{equation}
		C(\rho) \leq E(\phi) + \int_{\R^d} \phi \rho,
	\end{equation}
	hence $\phi$ is a KP and the proof of Lemma~\ref{lem:subgradient} is over.
\end{proof} 

\begin{remark}
	The optimal transport functional $\rho \mapsto C(\rho)$ is convex and lower semicontinuous --- as a functional from the set of positive and finite measures with fixed integer mass $N$. We can thus define its subdifferential at any $\rho$ as the set of \textbf{continuous} (and say bounded) functions $\phi$ such that precisely \eqref{eq:subdifferential} holds for every measure $\eta$ so that $\eta \geq 0$ and $\int_{\R^d} \eta = N \in \N^*$. So, in essence, Lemma~\ref{lem:subgradient} exactly says that continuous (and bounded) subgradients of the optimal transport functional are KPs. Nevertheless, we do not use this formalism since we want to consider KP which are to start with in the Lebesgue space $\mathcal{I}_\rho$ and therefore not necessarily continuous. We recall that the equivalence between subgradients of the optimal transport functional and the corresponding Kantorovich potentials is well-known, see e.g. \cite[Prop. 7.17]{santambrogio_optimal_2015}.
\end{remark}

A crucial ingredient in the proof of our main theorem will be the asymptotic behaviour of KPs:
\begin{lemma}\label{lem:up-bnd-asymp}
Let $\rho \in L^1(\R^d, \R_+)$ with $\int_{\R^d} \rho = N$, and let $\phi$ be any $c$-conjugated Kantorovich potential for the Kantorovich dual \eqref{kd}. Then, $\phi$ admits a limit $C_\phi$ everywhere at infinity and moreover $\phi \geq C_\phi$ everywhere. If furthermore $\rho$ has unbounded support, then we have the asymptotics
    \begin{equation}\label{eq:asymp}
		\phi(x) = \frac{N-1}{|x|^s} + C_\phi +o\Big(\frac{1}{|x|^s}\Big)
	\end{equation}
	in the limit $x \to \infty$ for $x$ inside the support of $\rho$.
\end{lemma}
\begin{proof}
    The existence of a well-defined limit at infinity is proved in \cite{lelotteExternalDualCharge2024}. It is furthermore shown there that the limit $C_\phi$ is negative and that
	\begin{equation}
		C_\phi = \inf_{x_1, \dots, x_{N-1} \in \R^d} \left\{c(x_1, \dots, x_{N-1}) - \sum_{i = 1}^{N-1} \phi(x_i)\right\}.
		\end{equation} 
    The fact that $\phi \geq C_\phi$ follows immediately from the non-negativity of the Riesz potential, so that 
    $$
        c(x, x_2, \dots, x_N) - \sum_{i = 2}^N \phi(x_i) \geq c(x_2, \dots, x_N) - \sum_{i = 2}^N \phi(x_i) 
    $$
        for all $x, x_2, \dots, x_N \in \R^d$ and thus the claim by taking the infimum with respect to $x_2, \dots, x_N \in \R^d$. As for the asymptotics \eqref{eq:asymp}, it is proved in \cite[Thm. 1]{lelotteAsymptoticsKantorovichPotential2025} --- in the case where the support of $\rho$ is unbounded and connected, but it remains valid for any target density $\rho$ with unbounded support for $c$-conjugated KPs, as noticed in \cite[Rem. 4]{lelotteAsymptoticsKantorovichPotential2025}. 
\end{proof}

\section{Proof of main theorem}\label{sec:proof-main-thm}
This section is dedicated to the proof of our main theorem, that is  Theorem~\ref{thm:main-thm}. Let $\rho \in L^1\cap L^{\cs}(\R^d)$ be a minimising density for the Lieb-Oxford bound \eqref{eq:LO-rho} at \textit{fixed} number of particles $N \in \N^*$, so that $\int_{\R^d} \rho = N$. Then, $\rho$ is \textit{a fortiori} a minimiser of the functional 
$$
F(\eta) := C(\eta) -\tfrac12 D(\eta) + \clo(s,d,N) \int_{\R^d} \eta^{\cs}.
$$
Hereafter, we write $\clo = \clo(s,d,N)$. In particular for any $\eta \in L^1 \cap L^{\cs}(\R^d)$ with $\int_{\R^d} \eta = N$ and all $\eps > 0$ it must be that $F(\rho + \eps (\eta - \rho)) \geq F(\rho) = 0$. The Hartree functional $\eta \mapsto \frac12 D(\eta)$ is differentiable on $L^1 \cap L^{\cs}(\R^d)$ and its functional derivative at a density $\eta$ is given by the Riesz potential $U^\eta$ generated by $\eta$, that is
\begin{equation}
	U^\eta(x) := \eta \ast |x|^{-s} = \int_{\R^d} \frac{\eta(y)}{|x - y|^s}\dd y.
\end{equation}
Furthermore, the local term $\eta \mapsto \int_{\R^d} \eta^{\cs}$ is also differentiable and its functional derivative at $\eta$ is given by $(\cs)\eta^{\frac sd}$. The optimal transport term $\eta \mapsto C(\eta)$ is \textit{not} differentiable in the usual way, but it is nevertheless convex with respect to $\eta$, so that we obtain
\begin{equation}\label{eq:subdiff}
	C(\eta) \geq C(\rho) + \int_{\R^d}(U^\rho - \beta_s\clo \rho^{\frac sd}) \dd (\eta - \rho)
\end{equation}
where we let $\beta_s := \cs$ for shortness. Let us write $$\phi_\rho := U^\rho - \beta_s \clo \rho^{\frac sd}.$$ Then, since the Riesz potential $U^\rho$ and $\rho^{\frac sd}$ are locally integrable, the function $\phi_\rho$ is also locally integrable. Therefore, we conclude from \eqref{eq:subdiff} that $\phi_\rho$ is a KP according to Lemma~\ref{lem:subgradient}. According to the second item of Theorem~\ref{thm:summary}, let $\phi$ be a $c$-conjugated and Lipschitz-continuous KP such that
  \begin{equation}\label{eq:ext-est-bis}
  	\phi_\rho + \frac{E(\phi_\rho)}{N} \leq \phi.
  \end{equation} 
  almost everywhere and such that (by optimality) we have equality almost everywhere on the support of $\rho$. In particular, this means that $\phi_\rho$ is (representable as) a Lipschitz-continuous function on the support of $\rho$. Furthermore, by continuity of $U^\rho$ outside of the support of $\rho$ (as follows readily from dominated convergence) we also have that $\phi_\rho$ is continuous there. We emphasise that for the moment, it is unknown whether or not the KP $\phi_\rho$ is actually continuous everywhere, that is whether or not it is (representable as a) continuous (function) across the boundary of its support of $\rho$

Let us write $c := \frac{E(\phi_\rho)}{N}$, so that we have
\begin{align}
		&U^\rho + c \leq \phi \quad \text{on} \,\, \R^d \setminus \supp(\rho) \label{eq:ext-est}\\
		&\phi_\rho + c = \phi \quad \text{on} \,\, \supp(\rho) \label{eq:int-est}
\end{align}
where the exterior estimate \eqref{eq:ext-est} is a direct consequence of \eqref{eq:ext-est-bis}.
From these optimality conditions, we actually have more regularity on our objects. First, let us prove the following lemma.
\begin{lemma}\label{lem:reg}
It holds that 
\begin{enumerate}
    \item \label{item:Urhocont} the Riesz potential $U^\rho$ is continuous everywhere on $\R^d$;
    \item \label{item:rhocont} the density $\rho$ is (representable as) a continuous function on $\R^d$ and vanishes on the boundary of its support.
\end{enumerate}
	
\end{lemma}
\begin{proof}[Proof of Lemma~\ref{lem:reg}]
	Let us first prove the first item. Let us remark that by the fact that $\rho \in L^{1 + s/d}(\R^d)$, the Riesz potential $U^\rho$ belongs to the Lebesgue space $L^{(d+s)d/s^2}(\R^d)$ according to \textit{Hardy--Littlewood--Sobolev} (HLS) \textit{inequality}, see \textit{e.g.} \cite{lieb_analysis_2001}. We note that the exponent $\cs$ is \textit{not} big enough to ensure that $U^\rho$ is continuous. On this matter, we recall that $U^\rho$ is continuous as soon as $\rho \in L^q(\R^d)$ -- or $\rho \in L^1 \cap L^q_{loc}(\R^d)$ -- with $q > \frac d{d-s}$ by Sobolev embedding. The idea to reach this critical exponent is to use the inner estimate \eqref{eq:int-est} to obtain more (local) integrability on the density $\rho$.
	
	 Indeed, from the inner estimate \eqref{eq:int-est} and the HLS inequality, we have that $\rho$ is actually \textit{locally} in $L^{1 + d/s}(\R^d)$ -- where here we use the fact that $\phi$ is uniformly bounded, \textit{cf.} the second item of Theorem~\ref{thm:summary}. According to what precedes, the Riesz potential $U^\rho$ will be automatically continuous as soon as $1 + \frac ds > \frac d {d - s}$. Although this is verified for the Coulomb potential in three space-dimension, that is for $d = 3$ and $s = 1$, it is neither verified for all $0 < s < d$ nor for the Coulomb potential in \textit{arbitrary} dimension, that this for $s = d-2$. Nevertheless, we can iterate the previous argument which reads in full generality as follows: \textit{If $\rho \in L^1 \cap L^q_{loc}(\R^d)$ then it is actually in $\rho \in L^1 \cap L^{q'}_{loc}(\R^d)$ where $q' := \tfrac{sq}{d - (d-s)q}$ by the HLS inequality and the inner estimate}. This makes sense as soon as $q < \tfrac{d}{d-s}$. We now show that, by iteration of this procedure, one obtains that $\rho \in L^q_{loc}(\R^d)$ for some $q > \tfrac d {d-s}$ -- hence continuity of $U^\rho$. Indeed, let us define $q_{n+1} = f(q_n)$ where $$f(q) := \frac{sq}{d - (d-s)q}$$ and $q_0 = \cs$. Letting $z_n := \tfrac1{q_n}$, we have $z_{n+1} = Az_n - B$ where $A = \tfrac d s$ and $B = \tfrac{d-s}{s}$ from which it follows elementarily that $z_n = A^{n}(z_0 - 1) + 1$. Since $z_0<1$, we have that $z_n$ is decreasing and $z_n \to -\infty$. Let us consider the last $n$ for which $z_n \geq 0$; in particular $Az_n-B = z_{n+1} < 0$, that is $z_n < \frac BA =\frac{ d-s}d$. We then conclude since $q_n= \frac 1{z_n} > \frac d{d-s}$.

     Let us then prove the second item. We consider the function $F := U^\rho - \phi + c$. According to what precedes, the function $F$ is continuous everywhere on $\R^d$. From the inner estimate \eqref{eq:int-est}, it holds that $F = \beta_s \clo \rho^{\frac sd}$ almost everywhere on the support of $\rho$.  This implies that $\rho$ is representable as a continuous function on its support. Furthermore, the exterior estimate \eqref{eq:ext-est} implies that $F \leq 0$ on $\R^d \setminus \supp(\rho)$. Thus, (the continuous representative of) $\rho$ vanishes on the boundary of its support.
\end{proof}

Using the previous lemma, we can show that, \textbf{provided the support of $\rho$ is unbounded}, we have $c \leq C_\phi$ where $C_\phi$ is the limit at infinity of the $c$-conjugated KP $\phi$ --- see Lemma~\ref{lem:up-bnd-asymp}. Indeed, let us consider a sequence $(x_n)_n \subset \supp(\rho)$ leaving all compact sets and that approaches the boundary of $\supp(\rho)$ from the inside as $n \to \infty$, so that $\rho(x_n) \rightarrow 0$ as $n \to \infty$ --- where we implicitly consider the continuous representative of the density $\rho$. This is possible thanks to Lemma~\ref{lem:reg}. It then holds that
    $$
    U^\rho(x_n) = \beta_s \clo \rho^{\frac sd}(x_n) + \phi(x_n) - c
    $$
   for all $n \in \N$ using the inner estimate \eqref{eq:int-est}. Taking (say) the supremum limit, we obtain
    \begin{equation}\label{eq:graal}
    0 \leq \limsup_{n \to \infty} U^\rho(x_n) = \lim_{n \to \infty} \beta_s \clo \rho^{\frac sd}(x_n) + \lim_{n \to \infty} \phi(x_n) - c = C_\phi - c
    \end{equation}
    and thus the aforementioned claim that $c \leq C_\phi$.    
    In fact, we will now show the stronger statement (again \textbf{provided the support of $\rho$ is unbounded}) that $\boxed{c = C_\phi}$.
    
    \medskip
 We remark that this would immediately follow from \eqref{eq:graal} if we could show that the Riesz potential $U^\rho$ vanishes everywhere at infinity. In all due respect, this need \textit{not} be true for a given $\rho$. For this to hold, one typically needs a bit of integrability on the density. For instance, if we can show that $\rho \in L^q(\R^d)$ for some $q > \frac{d}{d-s}$, then this would be veracious. Indeed, let us write 
$$
U^\rho(x) = \int_{|x-y|\leq R} \frac{\rho(y)}{|x -y|^s} \d y + \int_{|x-y| > R} \frac{\rho(y)}{|x-y|^s} \d y.
$$
for some $R > 0$. For the first term, Hölder's inequality yields
$$
\int_{|x-y|\leq R} \frac{\rho(y)}{|x -y|^s} \d y \leq \left( \int_{|x-y|\leq R} \rho^q \right)^{1/q} \left( \int_{|z| \leq R} |z|^{-sp} \d z\right)^{1/p}
$$
where $p$ is the conjugated exponent of $q$. Because $q > \frac d{d-s}$, we have $sp < d$ and thus the limit as $x$ goes to infinity of this term is zero,
$$
\lim_{|x|\to \infty} \int_{|x-y|\leq R} \frac{\rho(y)}{|x -y|^s} \d y = 0.
$$
As for the second term, we have roughly
$$
\int_{|x-y| > R} \frac{\rho(y)}{|x-y|^s} \d y \leq \frac1R \int_{\R^d} \rho = \frac NR.
$$
Therefore, we have
$$
0 \leq \limsup_{|x|\to\infty} U^\rho(x) \leq \frac{N}{R}
$$
and letting $R \to \infty$ yields the claim. It was shown in (the proof of) Lemma~\ref{lem:reg} that $\rho \in L^1 \cap L^q_{loc}(\R^d)$ for some $q > \frac d{d-s}$. We now claim that actually $\rho \in L^q(\R^d)$, thus that $U^\rho$ vanishes at infinity according to what precedes and ultimately that $c = C_\phi$ thanks to \eqref{eq:graal}. Indeed, by Lemma~\ref{lem:up-bnd-asymp}, we have $\phi \geq C_\phi$ everywhere. Now, by the inner estimate \eqref{eq:int-est}, it holds that 
$$
\phi_\rho = \phi - c \geq C_\phi - c \geq 0.
$$
where we have used the fact proved earlier that $c \leq C_\phi$. Recalling that $\phi_\rho = U^\rho -\beta_s \clo\rho^{\frac sd}$ we obtain thus that $\beta_s  \clo \rho^{\frac sd} \leq U^\rho$. As already discussed above in the proof of Lemma~\ref{lem:reg}, the \textit{Hardy-Littlewood-Sobolev inequality} implies $U^\rho$ is in the Lebesgue space $L^{(d+s)d/s^2}(\R^d)$ and thus from the aforementioned inequality we obtain that $\rho$ is in $L^{1 + \frac ds}(\R^d)$ \textit{globally} this time. The claim is therefore proved if $1 + \frac ds > \frac{d}{d-s}$, for instance for the Coulomb potential in three space-dimension. Otherwise, we can repeat \textit{exactly} the same rationale as in (the proof of) Lemma~\ref{lem:reg}, which yields that there exists $q > \frac d{d-s}$ so that $\rho \in L^q(\R^d)$. Indeed, we use (this time) that \textit{if $\rho \in L^q(\R^d)$ then it is actually in $\rho \in L^{q'}(\R^d)$ where $q' := \frac{sq}{d - (d-s)q}$ by HLS inequality and the inner estimate}. This is veracious since the uniformly bounded part $\phi - c$ is actually non-negative as shown above, so that the (global) integrability of the Riesz potential $U^\rho$ is transferred to the density $\rho$ at each iteration.

\medskip 
We can now prove the core of Theorem~\ref{thm:main-thm}.
\begin{proof}[Proof of Theorem~\ref{thm:main-thm}]
Let us show Theorem~\ref{thm:main-thm} by contradiction by assuming that the support of $\rho$ is unbounded and reaching a contradiction.

\medskip

Let us first show that the complement $\supp(\rho)^c$ \textit{cannot} be unbounded. Otherwise, using the exterior estimate \eqref{eq:ext-est}, we would have $U^\rho(x) \leq \phi(x) - C_\phi$ for all $x \in \supp(\rho)^c$, where we used the claim proved above that $c = C_\phi$. But, for all $0 < \eps < 1$, we have $U^\rho(x) \geq \frac{N-\eps}{|x|^s}$ for $x$ far enough. Now, using the asymptotics for the KP $\phi$ as stated in Lemma~\ref{lem:up-bnd-asymp}, far away it must be that
$$
\frac{N-\eps}{|x|^s} + o\Big( \frac1{|x|^s} \Big)\leq \frac{N-1}{|x|^s}+ o\Big( \frac1{|x|^s} \Big),
$$ 
an obvious contradiction. Therefore, \textit{reductio ad absurdum} the complement $\supp(\rho)^c$ is bounded, so that there exists a large enough ball $B_R$ such that $\supp(\rho)^c \subset B_R$ -- otherwise stated $B_R^c \subset \supp(\rho)$. But this entails that there exists (say) a connected component which contains $B_R^c$. Using that inner estimate \eqref{eq:int-est} and the asymptotic behavior of the KP $\phi$ as didacted by Lemma~\ref{lem:up-bnd-asymp}, it is ensured that
$$
\beta_s \clo \rho^{\frac sd} \geq \frac{1-\eps}{|x|^s} + o\left(\frac1{|x|^s}\right)
$$
as $x$ goes to infinity in that very component for all $0 < \eps < 1$. Therefore, $\rho$ cannot be integrable over $B_R^c$ hence on the connected component under consideration. Hence a contradiction since $\rho$ is a density.
\end{proof}

\appendix

\section{Discussion on the Kantorovich duality}\label{app:extension-buttazo}
In this appendix, we first prove that the non-standard formulation of the Kantorovich duality \eqref{kd-2} is veracious, meaning that the equality in \eqref{kd-2} holds. We then prove the two items of Theorem~\ref{thm:summary}.
\subsection{Proof of the strong Kantorovich duality \eqref{kd-2}}
It is known, see \textit{e.g.} \cite[Thm. 1.1]{buttazzo_continuity_2018} and also \cite[Rem. 2.1]{lelotteExternalDualCharge2024}, that 
$$
C(\rho) = \sup_{\phi} \left\{ \int_{\R^d} \phi \rho + \overline{E}(\phi)\right\}
$$
where the supremum runs over all (\textit{everywhere defined}) functions $\phi : \R^d \to \R$ so that $\int_{\R^d} |\phi|\rho < \infty$ and where we define 
$$
\overline{E}(\phi) := \inf_{x_1, \dots, x_N \in \R^d} \left\{c(x_1, \dots, x_N) - \sum_{i=1}^N \phi(x_i)\right\}.
$$
The difference here from \eqref{kd-2} is that the underlying admissibility constraint is required to hold \textit{everywhere}, that is not just \textit{almost} everywhere (which translates to the fact that above in the definition of $\overline{E}(\phi)$ an infimum is used whereas in \eqref{eq:E} it is an \textit{essential} infimum). Also note that the above is exactly the (standard) Kantorovich duality \eqref{kd} where the admissibility constraint \eqref{eq:constraint} has been made implicit. Evidently, for any $\phi : \R^d \to \R$, we have $\overline{E}(\phi) \leq E(\phi)$ and thus 
$$
C(\rho) \leq \sup_{\phi} \left\{ \int_{\R^d} \phi \rho + E(\phi)\right\} 
$$
(where again the supremum runs over $\rho$-integrable functions $\phi : \R^d \to \R$). As the quantity $E(\phi)$ only depends on the equivalence class of $\phi$ with respect to the (Lebesgue) almost everywhere equality, we therefore obtain
$$
C(\rho) \leq \sup_{\phi \in \mathcal{I}_\rho} \left\{ \int_{\R^d} \phi \rho + E(\phi)\right\}.
$$
where we recall that $\mathcal{I}_\rho = L^1_{loc}(\R^d) \cap L^1(\rho)$. To prove \eqref{kd-2}, let us show the converse inequality. Let $\phi \in \mathcal{I}_\rho$. Up to considering $\phi \leftarrow \phi + \frac{E(\phi)}{N}$, let us assume without loss of generality that $E(\phi) = 0$, otherwise stated that the admissibility constraint holds in the almost everywhere sense, so that in particular
$$
\phi(x_1) \leq c(x_1, \dots, x_N) - \sum_{i=2}^N \phi(x_i)
$$
for $x_1, \dots, x_N \in \R^d$ almost everywhere. Therefore, we have $\phi \leq \phi^1$ almost everywhere where we define the function $\phi^1 : \R^d \to \R$ as
$$
\phi^1(x_1) := \essinf_{x_2, \dots, x_N \in \R^d} \left\{c(x_1, \dots, x_N) - \sum_{i=2}^N \phi(x_i) \right\}.
$$
Note that $\phi^1$ is defined \textit{everywhere}. Furthermore, it verifies that
$$
\phi^1(x_1) + \sum_{i=2}^N \phi(x_i) \leq c(x_1, \dots, x_N)
$$
\textbf{for all $x_1 \in \R^d$} (to wit, \textit{everywhere}) and for all $x_2, \dots, x_N \in \R^d$ \textit{almost} everywhere. This entails that $\phi \leq \phi^2$ almost everywhere where this time we define the function $\phi^2 : \R^d \to \R$ as 
$$
\phi^2(x_2) := \inf_{x_1 \in \R^d} \essinf_{x_3, \dots, x_N \in \R^d} \left\{ c(x_1, \dots, x_N) - \phi^1(x_1) - \sum_{i=3}^N \phi(x_i)\right\}.
$$
This procedure can be further continued, where we construct $\phi^k : \R^d \to \R$ so that $\phi \leq \phi^k$ almost everywhere and
$$
\phi^k(x_k) := \inf_{x_1, \dots, x_{k-1} \in \R^d} \essinf_{x_{k+1}, \dots, x_N \in \R^d} \left\{ c(x_1, \dots, x_N) - \sum_{i=1}^{k-1}\phi^i(x_i) - \sum_{i=k+1}^N \phi(x_i)\right\}.
$$
But thus, at the end it holds that 
$$
\sum_{i=1}^N \phi^i(x_i) \leq c(x_1, \dots, x_N)
$$
\textbf{everywhere}, that is for all $x_1, \dots, x_N \in \R^d$. In particular, letting $\tilde{\phi} := \frac1N \sum_{i=1}^N \phi^i$, then $\tilde{\phi}$ verifies that $\overline{E}(\tilde{\phi}) \geq 0 = E(\phi)$ and, since $\phi \leq \tilde{\phi}$, we have $\int_{\R^d} \phi \rho \leq \int_{\R^d} \tilde{\phi} \rho$ --- provided the later integral makes sense, which is indeed the case, see \textit{infra}. Therefore, we have proved that for any $\phi \in \mathcal{I}_\rho$, we can build a function $\tilde{\phi} : \R^d \to \R$ defined everywhere so that 
$$
\int_{\R^d} \phi \rho  + E(\phi) \leq \int_{\R^d} \tilde{\phi} \rho + \overline{E}(\tilde{\phi}). 
$$
Thus
$$
\sup_{\phi \in \mathcal{I}_\rho} \left\{ \int_{\R^d} \phi \rho  + E(\phi)\right\} \leq \sup_{\phi} \left\{ \int_{\R^d} \phi \rho + \overline{E}(\phi) \right\} = C(\rho)
$$
where we yet emphasise that the second supremum runs over (everywhere defined) $\rho$-integrable functions $\phi : \R^d \to \R$. Thus, we are done with the proof of \eqref{kd-2}.

Note that considering the integral $\int_{\R^d} \tilde{\phi} \rho$ is legitimate since $\int_{\R^d} |\tilde{\phi}| \rho < \infty$. Indeed, it suffices to prove that $\int_{\R^d} |\phi^i| \rho < \infty$ for all $i=1, \dots, N$. First, recall that $\phi \leq \phi^1$. Furthermore, letting $f = f_{x_2, \dots, x_N} : \R^d \to \R$ be the function defined as 
$$
f(x) := c(x, x_2, \dots, x_N) - \sum_{i=2}^N \phi(x_i)
$$
then by definition of $\phi^1$ we have $\phi^1 \leq f$ almost everywhere for almost all $x_2, \dots, x_N \in \R^d$ --- beware that above in the definition of $f$, the ``$\phi$'' that appears is actually any representative function of $\phi$ seen as an element of the Lebesgue space $\mathcal{I}_\rho$. Moreover, we have
$$
\int_{\R^d} |f| \rho \leq \sum_{i=2}^N U^\rho(x_i) + N c(x_2, \dots, x_N) + N\sum_{i=2}^N |\phi(x_i)|.
$$
We may always select $x_2, \dots, x_N \in \R^d$ so that the right-hand side above is finite. Indeed, $\phi$ is $\rho$-integrable and thus finite $\rho$ almost everywhere. Furthermore, in the setting of our problem the potential $U^\rho$ is also finite $\rho$-almost everywhere. Indeed, we recall that $D(\rho) < \infty$ by the assumption that $\rho \in L^1 \cap L^{\cs}(\R^d)$ by the Hardy-Littlewood-Sobolev inequality. Since the Hartree energy rewrites as $D(\rho) = \int_{\R^d} U^\rho \rho$, the claim follows. Bottom line, $\phi \leq \phi^1 \leq f$ and both $f$ and $\phi$ are $\rho$-integrable, and therefore so is $\phi^1$. This rationale extends immediately to all the $\phi^i$'s.

\subsection{Proof of Theorem~\ref{thm:summary}}

We remark that from the equivalence of the two formulations of the Kantorovich duality shown in the preceding section, \textit{i.e.} \eqref{kd} and \eqref{kd-2}, the existence of a $c$-conjugated, Lipschitz (and uniformly bounded) KP for the later follows from the existence of such a KP for the former, a fact proved (again) in \cite{buttazzo_continuity_2018, colombo_continuity_2019}. This thus proves the first item of Theorem~\ref{thm:summary}.

Let us prove the second item of Theorem~\ref{thm:summary}. Let $\psi \in \mathcal{I}_\rho$ be any KP for \eqref{kd-2}. But then, according to the work done in the previous section, we know there exists a function $\widetilde{\psi} : \R^d \to \R$ admissible for the standard Kantorovich dual \eqref{kd} such that $\psi + \frac{E(\psi)}{N} \leq \widetilde{\psi}$ almost everywhere. By optimality, it must be that $\psi = \widetilde{\psi}$ $\rho$ almost everywhere and in particular $\widetilde{\psi}$ is a KP for \eqref{kd}. Then, we can appeal to \cite[Lem. 3.3]{buttazzo_continuity_2018} to ensure the existence of a $c-$conjugated KP $\phi$ so that $\widetilde{\psi} \leq \phi$. We recall that by $c$-conjugacy of $\phi$ it is meant that 
$$
\phi(x) = \inf_{x_2, \dots, x_N \in \R^d} \left\{ c(x, x_2, \dots, x_N) - \sum_{i = 2}^N \phi(x_i)\right\}
$$
It remains to show that $\phi$ is actually Lipschitz-continuous and uniformly bounded on $\R^d$. The later follows from \cite[Thm. 3.4]{buttazzo_continuity_2018} which shows that any $c$-conjugated KP for \eqref{kd} is necessarily uniformly bounded. From the fact that $\phi$ is indeed uniformly bounded, it is shown in (the proof of) \cite[Prop. 4.2]{lelotteExternalDualCharge2024} that the above $c$-conjugacy property still holds when one substitutes $c_\tau$ to the Riesz cost $c$, where $c_\tau$ is a truncated version of $c$ for some small enough $\tau > 0$, more precisely
$$
c_\tau(x_1, \dots, x_N) := \sum_{1 \leq i < j \leq N} \min\left\{\tau^{-s}, \frac{1}{|x_i - x_j|^s} \right\}.
$$
Now, since the function
$$
x \mapsto c_\tau(x, x_2, \dots, x_N) - \sum_{i = 2}^N \phi(x_i)
$$
is actually Lipschitz-continuous, any $c_\tau-$conjugated KP automatically inherits from this property and we are over. 

\printbibliography
\end{document}